\documentclass[a4paper,leqno,12pt]{amsart}
\usepackage{amsmath,amsthm}
\usepackage{amssymb}
\usepackage{xspace}
\usepackage{bm}
\usepackage{amstext}
\usepackage{amsfonts}
\usepackage{graphicx}
\usepackage[mathscr]{euscript}
\usepackage{amscd}
\usepackage{latexsym}
\usepackage{amssymb}
\usepackage{enumerate}
\usepackage{mathrsfs}
\usepackage[cmtip,all]{xy}
\usepackage{tikz}
\usetikzlibrary{shapes.geometric}
\begin{document}
%
%
\theoremstyle{plain}
\swapnumbers
    \newtheorem{thm}[figure]{Theorem}
    \newtheorem{prop}[figure]{Proposition}
    \newtheorem{lemma}[figure]{Lemma}
    \newtheorem{keylemma}[figure]{Key Lemma}
    \newtheorem{corollary}[figure]{Corollary}
    \newtheorem{fact}[figure]{Fact}
    \newtheorem{subsec}[figure]{}
    \newtheorem*{propa}{Proposition A}
    \newtheorem*{thma}{Theorem A}
    \newtheorem*{thmb}{Theorem B}
    \newtheorem*{thmc}{Theorem C}
    \newtheorem*{thmd}{Theorem D}
\theoremstyle{definition}
    \newtheorem{defn}[figure]{Definition}
    \newtheorem{example}[figure]{Example}
    \newtheorem{examples}[figure]{Examples}
    \newtheorem{notation}[figure]{Notation}
    \newtheorem{summary}[figure]{Summary}
\theoremstyle{remark}
        \newtheorem{remark}[figure]{Remark}
        \newtheorem{remarks}[figure]{Remarks}
        \newtheorem{warning}[figure]{Warning}
    \newtheorem{assume}[figure]{Assumption}
    \newtheorem{ack}[figure]{Acknowledgements}
\renewcommand{\thefigure}{\arabic{section}.\arabic{figure}}
%
%
%
\newenvironment{myeq}[1][]
{\stepcounter{figure}\begin{equation}\tag{\thefigure}{#1}}
{\end{equation}}
\newcommand{\myeqn}[2][]
{\stepcounter{figure}\begin{equation}
     \tag{\thefigure}{#1}\vcenter{#2}\end{equation}}
\newcommand{\mydiag}[2][]{\myeq[#1]\xymatrix{#2}}
\newcommand{\mydiagram}[2][]
{\stepcounter{figure}\begin{equation}
     \tag{\thefigure}{#1}\vcenter{\xymatrix{#2}}\end{equation}}
\newcommand{\mysdiag}[2][]
{\stepcounter{figure}\begin{equation}
     \tag{\thefigure}{#1}\vcenter{\xymatrix@R=15pt@C=20pt{#2}}\end{equation}}
\newcommand{\myrdiag}[2][]
{\stepcounter{figure}\begin{equation}
     \tag{\thefigure}{#1}\vcenter{\xymatrix@R=20pt@C=4pt{#2}}\end{equation}}
\newcommand{\mywdiag}[2][]
{\stepcounter{figure}\begin{equation}
     \tag{\thefigure}{#1}\vcenter{\xymatrix@R=20pt@C=12pt{#2}}\end{equation}}
\newcommand{\myzdiag}[2][]
{\stepcounter{figure}\begin{equation}
     \tag{\thethm}{#1}\vcenter{\xymatrix@R=5pt@C=30pt{#2}}\end{equation}}
%
\newenvironment{mysubsection}[2][]
{\begin{subsec}\begin{upshape}\begin{bfseries}{#2.}
\end{bfseries}{#1}}
{\end{upshape}\end{subsec}}
\newenvironment{mysubsect}[2][]
{\begin{subsec}\begin{upshape}\begin{bfseries}{#2\vsn.}
\end{bfseries}{#1}}
{\end{upshape}\end{subsec}}
\newcommand{\supsect}[2]
{\vspace*{-5mm}\quad\\\begin{center}\textbf{{#1}}\vsm.~~~~\textbf{{#2}}\end{center}}
\newcommand{\sect}{\setcounter{figure}{0}\section}
%
%
\newcommand{\wh}{\ -- \ }
\newcommand{\wwh}{-- \ }
\newcommand{\w}[2][ ]{\ \ensuremath{#2}{#1}\ }
\newcommand{\ww}[1]{\ \ensuremath{#1}}
\newcommand{\www}[2][ ]{\ensuremath{#2}{#1}\ }
\newcommand{\wwb}[1]{\ \ensuremath{(#1)}-}
\newcommand{\wb}[2][ ]{\ (\ensuremath{#2}){#1}\ }
\newcommand{\wref}[2][ ]{\ (\ref{#2}){#1}\ }
\newcommand{\wwref}[2]{\ (\ref{#1})-(\ref{#2})\ }
%
%
\newcommand{\hsp}{\hspace*{9 mm}}
\newcommand{\hs}{\hspace*{4 mm}}
\newcommand{\hsn}{\hspace*{1 mm}}
\newcommand{\hsm}{\hspace*{2 mm}}
\newcommand{\vsn}{\vspace{2 mm}}
\newcommand{\vs}{\vspace{5 mm}}
\newcommand{\vsm}{\vspace{3 mm}}
\newcommand{\vsp}{\vspace{9 mm}}
%
%
\newcommand{\hra}{\hookrightarrow}
\newcommand{\xra}[1]{\xrightarrow{#1}}
\newcommand{\xepic}[1]{\xrightarrow{#1}\hspace{-5 mm}\to}
\newcommand{\lora}{\longrightarrow}
\newcommand{\lra}[1]{\langle{#1}\rangle}
\newcommand{\llrra}[1]{\langle\langle{#1}\rangle\rangle}
\newcommand{\llrr}[2]{\llrra{#1}\sb{#2}}
\newcommand{\llrrp}[2]{\llrra{#1}'\sb{#2}}
\newcommand{\lrf}{\langle\langle f\lo{0,1}\rangle\rangle}
\newcommand{\lrfn}[1]{\lrf\sb{#1}}
\newcommand{\lras}[1]{\langle{#1}\rangle\sb{\ast}}
\newcommand{\lrau}[1]{\langle{#1}\rangle\sp{\ast}}
\newcommand{\vlam}{\vec{\lambda}}
\newcommand{\epic}{\to\hspace{-3.5 mm}\to}
\newcommand{\xhra}[1]{\overset{#1}{\hookrightarrow}}
\newcommand{\efp}{\to\hspace{-1.5 mm}\rule{0.1mm}{2.2mm}\hspace{1.2mm}}
\newcommand{\efpic}{\mbox{$\to\hspace{-3.5 mm}\efp$}}
\newcommand{\up}[1]{\sp{(#1)}}
\newcommand{\bup}[1]{\sp{[{#1}]}}
\newcommand{\lo}[1]{\sb{(#1)}}
\newcommand{\lolr}[1]{\sb{\lra{#1}}}
\newcommand{\bp}[1]{\sb{[#1]}}
\newcommand{\hfsm}[2]{{#1}\ltimes{#2}}
\newcommand{\sms}[2]{{#1}\wedge{#2}}
\newcommand{\rest}[1]{\lvert\sb{#1}}
%
%
\newcommand{\ab}{\operatorname{ab}}
\newcommand{\Aut}{\operatorname{Aut}}
\newcommand{\Cof}[1]{\operatorname{Cof}(#1)}
\newcommand{\Coker}{\operatorname{Coker}}
\newcommand{\colim}{\operatorname{colim}}
\newcommand{\comp}{\mbox{\sf comp}}
\newcommand{\Cone}{\operatorname{Cone}}
\newcommand{\csk}[1]{\operatorname{csk}\sp{#1}}
\newcommand{\diag}{\operatorname{diag}}
\newcommand{\ev}{\operatorname{ev}}
\newcommand{\exc}{\operatorname{ex}}
\newcommand{\Ext}{\operatorname{Ext}}
\newcommand{\Fib}{\operatorname{Fib}}
\newcommand{\fwd}{\operatorname{fwd}}
\newcommand{\gr}{\operatorname{gr}}
\newcommand{\ho}{\operatorname{ho}}
\newcommand{\hocofib}{\operatorname{hocofib}}
\newcommand{\hocolim}{\operatorname{hocolim}}
\newcommand{\holim}{\operatorname{holim}}
\newcommand{\Hom}{\operatorname{Hom}}
\newcommand{\inc}{\operatorname{inc}}
\newcommand{\Ker}{\operatorname{Ker}}
\newcommand{\Id}{\operatorname{Id}}
\newcommand{\Image}{\operatorname{Im}}
\newcommand{\md}{\operatorname{mid}}
\newcommand{\Obj}[1]{\operatorname{Obj}\,{#1}}
\newcommand{\op}{\sp{\operatorname{op}}}
\newcommand{\pt}{\operatorname{pt}}
\newcommand{\sk}[1]{\operatorname{sk}\sb{#1}}
\newcommand{\SL}[1]{\operatorname{SL}\sb{#1}(\ZZ)}
\newcommand{\Sq}[1]{\operatorname{Sq}\sp{#1}}
\newcommand{\Tot}{\operatorname{Tot}}
\newcommand{\uTot}{\underline{\Tot}}
%
%
\newcommand{\map}{\operatorname{map}}
%
%
\newcommand{\Hu}[3]{H\sp{#1}({#2};{#3})}
\newcommand{\Hus}[2]{\Hu{\ast}{#1}{#2}}
\newcommand{\HuF}[2]{\Hu{#1}{#2}{\Fp}}
\newcommand{\HuFs}[1]{\HuF{\ast}{#1}}
\newcommand{\HuT}[2]{\Hu{#1}{#2}{\Ft}}
\newcommand{\HuTs}[1]{\HuT{\ast}{#1}}
%
%
\newcommand{\Hi}[3]{H\sb{#1}({#2};{#3})}
\newcommand{\His}[2]{\Hi{\ast}{#1}{#2}}
\newcommand{\HiF}[2]{\Hi{#1}{#2}{\Fp}}
\newcommand{\HiZ}[2]{\Hi{#1}{#2}{\ZZ}}
\newcommand{\HiZp}[2]{\Hi{#1}{#2}{\Zp}}
\newcommand{\HiFs}[1]{\His{#1}{\Fp}}
\newcommand{\HiT}[2]{\Hi{#1}{#2}{\Ft}}
\newcommand{\HiTs}[1]{\HiT{\ast}{#1}}
%
%
\newcommand{\Ei}[3]{E\sb{#1}\sp{{#2},{#3}}}
\newcommand{\Eis}[2]{E\sb{#1}(#2)}
\newcommand{\Eu}[3]{E\sp{#1}\sb{{#2},{#3}}}
\newcommand{\Eus}[1]{E\sp{#1}}
\newcommand{\Euz}[1]{E\sp{0}({#1})}
%
%
\newcommand{\A}{\mathcal{A}}
\newcommand{\tA}{\widetilde{A}}
\newcommand{\wA}{\widehat{A}}
\newcommand{\B}{\mathcal{B}}
\newcommand{\C}{\mathcal{C}}
\newcommand{\D}{\mathcal{D}}
\newcommand{\wD}{\widehat{D}}
\newcommand{\E}{\mathcal{E}}
\newcommand{\wE}{\widehat{E}}
\newcommand{\cF}[1]{\mathcal{F}\sp{#1}}
\newcommand{\wf}{\widehat{f}}
\newcommand{\wI}{\widehat{I}}
\newcommand{\KK}[1]{{\mathcal{K}}\sb{#1}}
\newcommand{\M}{\mathcal{M}}
\newcommand{\Map}{{\EuScript Map}}
\newcommand{\MA}{\Map\sb{\A}}
\newcommand{\MAB}{\Map\sb{\A}\sp{\B}}
\newcommand{\MB}{\Map\sp{\B}}
\newcommand{\OO}{\mathcal{O}}
\newcommand{\cP}[1]{\mathcal{P}\sp{#1}}
\newcommand{\PP}[2]{\cP{#1}\sb{#2}}
\newcommand{\hP}[2]{\widehat{\mathcal{P}}\sp{#1}\sb{#2}}
\newcommand{\Ss}{\mathcal{S}}
\newcommand{\Sa}{\Ss\sb{\ast}}
\newcommand{\U}{\mathcal{U}}
\newcommand{\eW}{{\EuScript W}}
\newcommand{\eX}{{\EuScript X}}
\newcommand{\eY}{{\EuScript Y}}
\newcommand{\eZ}{{\EuScript Z}}
%
%
\newcommand{\hy}[2]{{#1}\text{-}{#2}}
\newcommand{\Alg}[1]{{#1}\text{-}{\mbox{\sf Alg}}}
\newcommand{\Pa}[1][ ]{$\Pi$-algebra{#1}}
\newcommand{\PAlg}{\Alg{\Pi}}
\newcommand{\pis}{\pi\sb{\ast}}
\newcommand{\gS}[1]{{\EuScript S}\sp{#1}}
\newcommand{\Mod}[1]{{#1}\text{-}{\mbox{\sf Mod}}}
\newcommand{\us}{u\sp{\ast}}
\newcommand{\Set}{\mbox{\sf Set}}
\newcommand{\Seta}{\Set\sb{\ast}}
\newcommand{\Cat}{\mbox{\sf Cat}}
\newcommand{\Ch}{\mbox{\sf Ch}}
\newcommand{\Grp}{\mbox{\sf Gp}}
\newcommand{\OC}{\hy{\OO}{\Cat}}
\newcommand{\SC}{\hy{\Ss}{\Cat}}
\newcommand{\SaC}{\hy{\Sa}{\Cat}}
\newcommand{\SO}{(\Ss,\OO)}
\newcommand{\SaO}{(\Sa,\OO)}
\newcommand{\SOC}{\hy{\SO}{\Cat}}
\newcommand{\SaOC}{\hy{\SaO}{\Cat}}
\newcommand{\Top}{\mbox{\sf Top}}
\newcommand{\Topa}{\Top\sb{\ast}}
%
%
\newcommand{\FF}{\mathbb F}
\newcommand{\Fp}{\FF\sb{p}}
\newcommand{\Ft}{\FF\sb{2}}
\newcommand{\Fq}{\FF\sb{q}}
\newcommand{\NN}{\mathbb N}
\newcommand{\QQ}{\mathbb Q}
\newcommand{\ZZ}{\mathbb Z}
\newcommand{\Zp}{\ZZ\sb{(p)}}
\newcommand{\Zt}{\ZZ\sb{(2)}}

%
%
\newcommand{\bA}{{\mathbf A}}
\newcommand{\bB}{{\mathbf B}}
\newcommand{\bC}{{\mathbf C}}
\newcommand{\bD}{{\mathbf D}}
\newcommand{\cD}{D}
\newcommand{\bDel}{\mathbf{\Delta}}
\newcommand{\Del}[1]{\bDel\sp{#1}}
\newcommand{\Dels}[2]{\Del{#1}\lo{#2}}
\newcommand{\Dlt}[1]{\Delta[{#1}]}
\newcommand{\bE}{{\mathbf E}}
\newcommand{\be}[1]{{\mathbf e}\sp{#1}}
\newcommand{\bF}{{\mathbf F}}
\newcommand{\Fv}[1]{\bF\bp{#1}}
\newcommand{\Fk}[1]{F\sp{#1}}
\newcommand{\Fpp}{\hspace{0.5mm}'\hspace{-0.7mm}F}
\newcommand{\Fpk}[1]{\hspace{0.5mm}'\hspace{-0.7mm}F\sp{#1}}
\newcommand{\Fn}[2]{F\sp{#1}\bp{#2}}
\newcommand{\bG}{{\mathbf G}}
\newcommand{\Gv}[1]{\bG\bp{#1}}
\newcommand{\hGv}[1]{\widehat{\bG}\bp{#1}}
\newcommand{\Gn}[2]{G\sp{#1}\bp{#2}}
\newcommand{\hGn}[2]{\widehat{G}\sp{#1}\bp{#2}}
\newcommand{\tg}[1]{\widetilde{g}\sp{#1}}
\newcommand{\bH}{{\mathbf H}}
\newcommand{\wH}{\widehat{H}}
\newcommand{\Hv}[1]{\bH\bp{#1}}
\newcommand{\hHv}[1]{\wH\bp{#1}}
\newcommand{\Hn}[2]{H\sp{#1}\bp{#2}}
\newcommand{\tH}[1]{\widetilde{H}\sp{#1}}
\newcommand{\tHn}[2]{\tH{#1}\bp{#2}}
\newcommand{\bi}{{\mathbf i}}
\newcommand{\bj}{{\mathbf j}}
\newcommand{\bK}{{\mathbf K}}
\newcommand{\KP}[2]{\bK({#1},{#2})}
\newcommand{\KR}[1]{\KP{R}{#1}}
\newcommand{\KZ}[1]{\KP{\ZZ}{#1}}
\newcommand{\KZp}[1]{\KP{\Zp}{#1}}
\newcommand{\KZt}[1]{\KP{\Zt}{#1}}
\newcommand{\KF}[1]{\KP{\Fp}{#1}}
\newcommand{\bM}[1]{{\mathbf M}\sp{#1}}
\newcommand{\bP}{{\mathbf P}}
\newcommand{\bS}[1]{{\mathbf S}\sp{#1}}
\newcommand{\bT}{\mathbf{\Theta}}
\newcommand{\TA}{\bT\sp{\A}}
\newcommand{\TB}{\bT\sb{\B}}
\newcommand{\ThB}{\Theta\sb{\B}}
\newcommand{\TAB}{\bT\sp{\A}\sb{\B}}
\newcommand{\TR}{\Theta\sb{R}}
\newcommand{\TRl}{\Theta\sb{R}\sp{\lambda}}
\newcommand{\bU}{{\mathbf U}}
\newcommand{\bV}{{\mathbf V}}
\newcommand{\bW}{{\mathbf W}}
\newcommand{\bX}{{\mathbf X}}
\newcommand{\hX}{\widehat{\bX}}
\newcommand{\bY}{{\mathbf Y}}
\newcommand{\hY}{\widehat{\bY}}
\newcommand{\bZ}{{\mathbf Z}}
%
%
\newcommand{\bdel}{\bar{\delta}}
\newcommand{\eps}[1]{\epsilon\sb{#1}}
\newcommand{\gam}[1]{\gamma\sb{#1}}
\newcommand{\iot}[1]{\iota\sb{#1}}
\newcommand{\vare}{\varepsilon}
\newcommand{\ett}[1]{\eta\sb{#1}}
\newcommand{\blam}{\bar{\lambda}}
\newcommand{\wvarp}{\widehat{\varphi}}
\newcommand{\bn}{[\mathbf{n}]}
\newcommand{\wrho}{\widehat{\rho  }}
\newcommand{\tS}{\widetilde{\Sigma}}
\newcommand{\wsig}{\widehat{\sigma}}
%
%
\newcommand{\Cd}[1]{\bC[{#1}]\sb{\bullet}}
\newcommand{\oT}[1]{\overline{T}\sb{#1}}
\newcommand{\Vd}{V\sb{\bullet}}
\newcommand{\oV}[1]{\overline{V}\sb{#1}}
\newcommand{\Wd}{\bW\sb{\bullet}}
\newcommand{\oW}[1]{\overline{\bW}\sb{#1}}
%
%
\newcommand{\cH}[3]{H\sp{#1}({#2};{#3})}
\newcommand{\HR}[2]{\cH{#1}{#2}{R}}
\newcommand{\HsR}[1]{\HR{\ast}{#1}}
\newcommand{\HF}[2]{\cH{#1}{#2}{\Fp}}
\newcommand{\HsF}[1]{\HF{\ast}{#1}}
%
%
\newcommand{\ma}[1][ ]{mapping algebra{#1}}
\newcommand{\Ama}[1][ ]{$\A$-mapping algebra{#1}}
\newcommand{\ABma}[1][ ]{$\A$-$\B$-bimapping algebra{#1}}
\newcommand{\Bma}[1][ ]{$\B$-dual mapping algebra{#1}}
\newcommand{\Tma}[1][ ]{$\bT$-mapping algebra{#1}}
\newcommand{\lin}[1]{\{{#1}\}}
\newcommand{\fff}{\mathfrak{f}}
\newcommand{\fM}{\mathfrak{M}}
\newcommand{\fMA}{\fM\sp{\A}}
\newcommand{\fMB}{\fM\sb{\B}}
\newcommand{\fMAB}{\fM\sp{\A}\sb{\B}}
\newcommand{\fMT}{\fM\sb{\bT}}
\newcommand{\fMR}{\fM\sb{R}}
\newcommand{\fT}{\mathfrak{T}}
\newcommand{\fTd}{\fT\sb{\bullet}}
\newcommand{\fV}{\mathfrak{V}}
\newcommand{\fVd}{\fV\sb{\bullet}}
\newcommand{\fVu}{\fV\sp{\bullet}}
\newcommand{\fW}{\mathfrak{W}}
\newcommand{\fWd}{\fW\sb{\bullet}}
\newcommand{\fX}{\mathfrak{X}}
\newcommand{\fY}{\mathfrak{Y}}
\newcommand{\fZ}{\mathfrak{Z}}
\newcommand{\fin}{\operatorname{fin}}
\newcommand{\init}{\operatorname{init}}
\newcommand{\vf}{v\sb{\fin}}
\newcommand{\vi}{v\sb{\init}}
%
%
%
\title{Note on Toda brackets}
\author{Samik Basu, David Blanc, and Debasis Sen}
\address{Stat-Math Unit, Indian Statistical Institute, Kolkata 700108, India}
\email{samikbasu@isical.ac.in}
\address{Department of Mathematics\\ University of Haifa\\ 3498838 Haifa\\ Israel}
\email{blanc@math.haifa.ac.il}
\address{Department of Mathematics \& Statistics\\
Indian Institute of Technology, Kanpur\\ Uttar Pradesh 208016\\ India}
\email{debasis@iitk.ac.in}

\date{\today}


\begin{abstract}
  We provide a general definition of Toda brackets in a pointed model categories, show
  how they serve as obstructions to rectification, and explain their relation to the
  classical stable operations.
\end{abstract}

\maketitle

\setcounter{section}{0}

%
%
\section*{Introduction}
\label{cint}

Toda brackets, defined by Toda in \cite{TodG}, play an important role in homotopy
theory both for their original purpose of calculating homotopy groups in \cite{TodC},
and because they serve as differentials in spectral sequences (see \cite{AdHI} and
\cite{BBlaH,CFranH}). The notion has been generalized in several ways
(see, e.g., \cite{BBGondH,BJTurnC,GWalkL}), not all of which agree.

In this note we provide a definition of higher Toda brackets in a general
pointed model category $\C$, show how these appear as the successive obstructions to
strictifying certain diagrams (namely, chain complexes in the homotopy category
\w[)]{\ho\C} \wwh see Theorem \ref{trect} below \wh and explain to connection with
the traditional stable description in terms of filtered complexes.

\begin{ack}
The research of the second author was supported by Israel Science
Foundation grant 770/16.
\end{ack}

%
%
\sect{Higher Toda brackets}
\label{ctoda}

We provide a variant of the definition of higher Toda brackets sketched in
\cite[\S 7]{BJTurnC} which brings out clearly their connection to rectification
of linear diagrams.

\begin{defn}\label{dtodadiag}
Let $\C$ be a cofibrantly generated left proper pointed model category.
A \emph{Toda diagram}  of length $n$ for $\C$ is a diagram in the homotopy category
\w{\ho\C} of the form
\begin{myeq}\label{eqtodadiag}
A\sb{0}~\xrightarrow{[f\sb{0}]}~A\sb{1}~\xrightarrow{[f\sb{1}]}~A\sb{2}~\to~\dotsc~\to~
A\sb{n-1} ~\xrightarrow{[f\sb{n-1}]}~A\sb{n}
\end{myeq}
\noindent with \w{[f\sb{k}]\circ[f\sb{k-1}]=0} (the strict zero map) for each \w[.]{1\leq k<n}

A \emph{strictification} of \wref{eqtodadiag} is a diagram in $\C$ of the form
\begin{myeq}\label{eqstodadiag}
A'\sb{0}~\xrightarrow{f'\sb{0}}~A'\sb{1}~\xrightarrow{f'\sb{1}}~A'\sb{2}~\to~\dotsc~\to~
A'\sb{n-1} ~\xrightarrow{f'\sb{n-1}}~A'\sb{n}
\end{myeq}
\noindent with \w{f'\sb{k}\circ f'\sb{k-1}=0} for each \w[,]{1\leq k<n} which is weakly
equivalent to a lift
of \wref[.]{eqtodadiag}
\end{defn}

To obtain a homotopy meaningful description of such strictifications, we shall need:

\begin{defn}\label{detodadiag}
  Let \w{I\sp{n}\sb{2}} denote the lattice of subsets of \w[,]{\bn=\{1,2,\dotsc,n\}} which
  we shall think of as an $n$-dimensional cube with vertices labelled by
  \w{J=(\vare\sb{1},\vare\sb{2},\dotsc,\vare\sb{n})} with each \w[,]{\vare\sb{i}\in\{0,1\}}
  so $J$ is the characteristic function of a subset of  $\bn$.
The cube is partially ordered in the usual way, with a unique map \w{J\to J'} whenever
\w{\vare\sb{i}\leq\vare'\sb{i}} for all\w[.]{1\leq i\leq n} Write \w{J\sb{k}} for
the vertex labelled by $k$ ones followed by \w{n-k} zeros \wb[.]{k=0,1,\dotsc,n}

Note that any strictification \wref{eqstodadiag} extends uniquely to a diagram
\w[,]{D:I\sp{n}\sb{2}\to\C} in which \w[,]{D(J\sb{k})=A'\sb{k}} $D$ assigns the map \w{f'\sb{k}} to the
unique map \w{J\sb{k}\to J\sb{k+1}} in \w[,]{I\sp{n}\sb{2}} and $D$ sends all other vertices (and thus
all other maps) to $0$.

The category \w{\C\sp{I\sp{n}\sb{2}}} of all ``cube-shaped'' diagrams in $\C$ has
a cofibrantly generated model category structure, in which the weak equivalences
and fibrations are defined objectwise (see \cite[Theorem 11.6.1]{PHirM}).
Any diagram \w{\wD:I\sp{n}\sb{2}\to\C} weakly equivalent to the diagram $D$
just described will be called an \emph{enhanced strictification}
of \wref[.]{eqtodadiag}  In particular, when $\wD$ is a cofibrant replacement
for $D$ in the model category mentioned above, we call it a
\emph{cofibrant enhanced strictification}.
\end{defn}

\begin{remark}\label{retodadiag}
The cofibrations in the model category \w{\C\sp{I\sp{n}\sb{2}}} are not easy to identify explicitly.
However, because the indexing category is directed in a strong sense, the cofibrant diagrams
\w{B:I\sp{n}\sb{2}\to\C} may be described inductively by filtering \w{I\sp{n}\sb{2}} by composition
length: \w[.]{F\sb{0}\subset F\sb{1}\subset\dotsc\subset F\sb{n}=I\sp{n}\sb{2}} We start with
\w{F\sb{0}} consisting only of the initial object \w[;]{J\sb{0}} then \w{F\sb{k+1}} is obtained
  inductively from \w{F\sb{k}} by adding all indecomposable maps in \w{I\sp{n}\sb{2}} from objects in
  \w[,]{F\sb{k}} together with their targets.

  We then have a recursive definition of what it means for a diagram $B$ to be cofibrant, by requiring
  that   \w{B\rest{J\sb{k}}} be cofibrant for each $k$ (for \w[,]{k=0} this
just means ensuring \w{B(J\sb{0})} is cofibrant in $\C$). and then ensuring that the natural map
\w{\colim\sb{F\sb{k}}\,B\to B(J)} is a cofibration for each \w{J\in F\sb{k+1}\setminus F\sb{k}}
(cf.\ \cite[\S 2.14]{BJTurnH}).
In particular, this ensures that all morphisms in \w{I\sp{n}\sb{2}} are taken by $B$ to cofibrations
in $\C$.
\end{remark}

\begin{defn}\label{dextodadiag}
  Now let \w{I\sp{n}\sb{3}} denote the $n$-dimensional cube of side length $2$,
  with vertices labelled as before by
  \w[,]{J=(\vare\sb{1},\vare\sb{2},\dotsc,\vare\sb{n})} but now each
\w[.]{\vare\sb{i}\in\{0,1,2\}} Again, \w{I\sp{n}\sb{3}} is partially ordered as above.

Note that any cofibrant enhanced strictification \w{\wD:I\sp{n}\sb{2}\to\C} has a unique
extension to an \emph{extended diagram} \w[,]{E:I\sp{n}\sb{3}\to\C} with the following
property:

For each \w{1\leq k\leq n} and
\w[,]{\vare\sb{1},\dotsc,\vare\sb{k-1},\vare\sb{k+1}\dotsc,\vare\sb{n}\in\{0,1\}}
set:
\begin{myeq}\label{eqextodadiag}
  \begin{cases}
    J'&=(\vare\sb{1},\dotsc,\vare\sb{k-1},0,\vare\sb{k+1}\dotsc,\vare\sb{n})\\
    J''&=(\vare\sb{1},\dotsc,\vare\sb{k-1},1,\vare\sb{k+1}\dotsc,\vare\sb{n})\\
    J'''&=(\vare\sb{1},\dotsc,\vare\sb{k-1},2,\vare\sb{k+1}\dotsc,\vare\sb{n}).
    \end{cases}
\end{myeq}
\noindent We then have a strict cofibration sequence:
\begin{myeq}\label{eqcofseq}
E(J')~\hra~ E(J'')~\epic~E(J''')~.
\end{myeq}
\noindent That is,  \w{E(J')\hra E(J'')} is a cofibration in $\C$ (see \S \ref{rcofsq}),
and \w{E(J'')} is the colimit of \w[.]{\ast\leftarrow E(J')\to E(J'')}

This defines $E$, by functoriality of strict cofibration sequences.
In fact, \wref{eqcofseq} will be a strict (and also homotopy) cofibration sequence for every
\w{1\leq k\leq n} and
\w[,]{\vare\sb{1},\dotsc,\vare\sb{k-1},\vare\sb{k+1}\dotsc,\vare\sb{n}\in\{0,1,2\}}
by the \w{3\times 3} Lemma (the cofibration sequence version of
\cite[XII, Lemma 3.4]{MacLH}, which actually follows from it by working in
simplicial groups). Here we use the description of $\wD$ in Remark \ref{retodadiag}
and the fact that $\C$ is left proper.

The original \wwb{2\times\dotsc\times 2}cube diagram \w{\wD=E\rest{I\sp{n}\sb{2}}} will
be called the \emph{generating cube} for $E$.
\end{defn}

\begin{notation}\label{nextodadiag}
  For any index \w{J=(\vare\sb{1},\vare\sb{2},\dotsc,\vare\sb{n})}
  \wb[,]{\vare\sb{i}\in\{0,1,2\}} we shall use the following terminology:
\begin{enumerate}
\renewcommand{\labelenumi}{(\arabic{enumi})~}
\item The longest initial segment of $J$ with \w{\vare\sb{i}\in\{1,2\}}  will be called
  the \emph{marker} of $J$ (it may be empty, or all of $J$), and denoted by \w[.]{M(J)}
\item The length of the final segment of \w{M(J)} consisting of digits 2 only will be called the
  \emph{stage} of $J$, and denoted by \w[.]{\sigma(J)}
\item The complementary final segment of $J$ after \w{M(J)} (necessarily starting with a 0, if
  nonempty) is called the \emph{remainder} of $J$, and denoted by \w[.]{R(J)} The number of
  digits 2 in \w{R(J)} will be denoted by \w[.]{r(J)}
\end{enumerate}
\end{notation}

\begin{example}\label{egnextodadiag}
  For \w{J=(1221122202012012)} we have marker \w[.]{M(J)=(12211222)} Its stage is
  \w{\sigma(J)=3} (the length of the sequence \w{222} at the end of \w[),]{M(J)}
  and the remainder is \w[,]{R(J)=02012012} with \w[.]{r(J)=3}
\end{example}

We note the following obvious facts:

\begin{lemma}\label{lextodadiag}
  For any index \w[,]{J=(\vare\sb{1},\vare\sb{2},\dotsc,\vare\sb{n})} the extended diagram
  \w{E:I\sp{n}\sb{3}\to\C}  of any cofibrant enhanced strictification
  \w{\wD:I\sp{n}\sb{2}\to\C} as above has the following properties
\begin{enumerate}
\renewcommand{\labelenumi}{(\alph{enumi})~}
\item If \w{R(J)} has at least one digit 1, then \w{E(J)} is weakly contractible.
\item If \w{R(J)} has no $1$'s, then \w{E(J)} is weakly equivalent to
\w[,]{\Sigma\sp{r(J)}E(J')} where \w{J'} is obtained from $J$ by replacing all $2$'s in
\w{R(J)} by digits 0.
\item In particular, if \w{\sigma(J)=0} and \w{R(J)} has no digits 1, then
  \w[,]{E(J)\simeq \Sigma\sp{r(J)}A\sb{k}} where $k$ is the length of \w[.]{M(J)}
  Thus if \w{\sigma(J)=0} and \w{R(J)} has only digits 0, then \w[.]{E(J)\simeq A\sb{k}}
  In particular, \w[,]{E(J\sb{k})\simeq A\sb{k}} as in Definition \ref{detodadiag}.
  Here \w{\Sigma\sp{r(J)}A} is the \ww{r(J)}-fold suspension of $A$ (cf.\ \S\ref{nextodadiag}(3)).
\item If \w{\sigma(J)>0} and \w{R(J)} has no digits 1, then \w{E(J)} is the cofiber of a map
\w[,]{E(J\sb{0})\to E(J\sb{1})} where \w{\sigma(J\sb{0})=0} and \w[.]{\sigma(J\sb{1})=\sigma(J)-1}
Thus we can think of \w{\sigma(J)} as the ``cone length'' of \w{E(J)}  with respect to the objects
\w{A\sb{0},\dotsc,A\sb{n}} and their suspensions.
\end{enumerate}
\end{lemma}

See the diagram in Example \ref{scnth} below for an illustration of these properties.

\begin{defn}\label{dhtb}
  We now explain how to associate to a Toda diagram of length \w{n+1}
  (Definition \ref{dtodadiag}), with
certain additional data, the Toda bracket \w[,]{\lra{f\sb{0},\dotsc,f\sb{n}}}
which serves as the (last) obstruction to strictifying the Toda diagram \wref[,]{eqtodadiag}
for \w{n\geq 1} (cf.\ \cite[\S 7]{BJTurnC}).

The construction is inductive, and a necessary condition
in order for it to be defined is that the Toda brackets \w{\lra{f\sb{0},\dotsc,f\sb{n-1}}}
and \w[,]{\lra{f\sb{1},\dotsc,f\sb{n}}}  associated respectively to the initial and final segments
of \wref{eqtodadiag} of length $n$, must vanish.

Vanishing of \w{\lra{f\sb{0},\dotsc,f\sb{n-1}}}  implies that we may choose a strictification
\w{D:I\sp{n}\sb{2}\to\C} of the initial segment, with the corresponding cofibrant enhanced
strictification \w{\wD:I\sp{n}\sb{2}\to\C} and extended diagram \w[.]{E:I\sp{n}\sb{3}\to\C}
By Lemma \ref{lextodadiag}(c), we may choose a representative
\w{f\sb{n}:E(J\sb{n})\to A\sb{n+1}} for the last term \w{[f\sb{n+1}]} in the diagram
\wref{eqtodadiag} (the construction does not actually depend on the choice of \w[).]{f\sb{n}}

The \emph{data} for  \w{\lra{f\sb{0},\dotsc,f\sb{n}}} consists of the weak homotopy
type of $D$
in \w[,]{\C\sp{I\sp{n}\sb{2}}} together with (the homotopy class of) a map
\w{\varphi\sb{n}:E(1,2,\dotsc, 2)\to A\sb{n+1}} whose restriction to
\w{E(1,2,\dotsc,2,1)\simeq A\sb{n}} represents \w[.]{[f\sb{n}]}
Note that \w{E( 1,2,\dotsc, 2)} has ``cone length'' \w[,]{n-1}  by Lemma \ref{lextodadiag}(d).
The class of  \w{\varphi\sb{n}} is determined by a choice of nullhomotopy for the
(inductively determined) value of \w[.]{\lra{f\sb{1},\dotsc,f\sb{n}}}

Since \w{E(0,2,\dotsc,2)\simeq\Sigma\sp{n-1}A\sb{0}} by Lemma \ref{lextodadiag}(c),
and precomposing \w{\varphi\sb{n}} with \w{\alpha\sb{n}:E(0,2,\dotsc,2)\to E(1,2,\dotsc,2)}
yields the \emph{value} in \w{[\Sigma\sp{n-1}A\sb{0},\,A\sb{n+1}]} of the Toda bracket
(associated to the data \w[).]{\lra{D,\varphi\sb{n}}}

Formally, the \emph{Toda bracket} \w{\lra{f\sb{0},\dotsc,f\sb{n}}} is the subset of 
\w{[\Sigma\sp{n-1}A\sb{0},\,A\sb{n+1}]} consisting of all such values. The differences
between the various values together constitute the \emph{indeterminacy} of
\w[.]{\lra{f\sb{0},\dotsc,f\sb{n}}} However, we shall not be concerned with these two notions
here: essentially, we are only interested in the question of whether certain elements in
\w{[\Sigma\sp{n-1}A\sb{0},\,A\sb{n+1}]} can be obtained as the value of some Toda bracket.

Since
$$
E(0,2,\dotsc,2)~\xra{\alpha\sb{n}}~E(1,2,\dotsc,2)~\epic~E(2,2,\dotsc,2)
$$
\noindent is a (homotopy) cofibration sequence, if
\w[,]{\varphi\sb{n}\circ\alpha\sb{n}\sim 0} we can choose an extension
\w{\psi\sb{n}:E(2 ,2,\dotsc, 2)\to A\sb{n+1}} of \w[,]{\varphi\sb{n}} up to homotopy.
\end{defn}

See the diagrams in Examples \ref{scntw} and \ref{scnth} below for an illustration
of these properties.

\begin{lemma}\label{lhtb}
  A choice of \w{\psi\sb{n}:E(2 ,2,\dotsc, 2)\to A\sb{n+1}} extending
  \w{\varphi\sb{n}:E(1,2,\dotsc, 2)\to A\sb{n+1}}
as above yields a strictification for
\w{A\sb{0}\xra{f\sb{0}}\dotsc \xra{f\sb{n}} A\sb{n+1}} extending
the class of $D$.
\end{lemma}

\begin{proof}
The strictification is given by
$$
E(0,0,\dotsc,0)\to E(1,0,\dotsc,0)\to E(2,1,0,\dotsc,0)\to\dotsc
E(2,2,\dotsc,2)\xra{\psi\sb{n}}A\sb{n+1}.
$$
\noindent Since each adjacent composition is of the form \wref{eqcofseq} for
\wref[,]{eqextodadiag} it is strictly zero. The fact that \w{E(2,\dotsc,2,1,0,\dotsc,0)}
(with \w{k-1} digits 2) is weakly equivalent to \w{A\sb{k}} follows from Lemma \ref{lextodadiag},
which implies the map out of \w{E(2,\dotsc,2,1,0,\dotsc,0)} represents \w{f\sb{k}} by the
original construction of $D$.
\end{proof}

\begin{remark}\label{rhtb}
Note that the ``middle cube'' \w{I\sp{n-1}\sb{\md}} of \w[,]{I\sp{n-1}\sb{3}} consisting of
the vertices indexed by \w{J=(\vare\sb{1},\dotsc,\vare\sb{n-1})} with \w[,]{\vare\sb{1}=1}
corresponds to the segment \w{A\sb{1}\xra{f\sb{1}}\dotsc \xra{f\sb{n-1}}A\sb{n}} of
\wref{eqtodadiag} (that is, \w{I\sp{n-1}\sb{\md}} is a cofibrant enhanced strictification
of this segment in the sense of \S \ref{detodadiag}).

Applying Definition \ref{dhtb} to \w{E\rest{I\sp{n-1}\sb{\md}}} we see that in addition to the
choices encoded in the given strictification of this segment , the data we obtain consists of
\w{\varphi'\sb{n-1}:E(1,1,2,\dotsc, 2)\to A\sb{n+1}} whose restriction to
\w{E(1,1,2,\dotsc,2,1)\simeq A\sb{n}} again represents \w[.]{[f\sb{n}]} The corresponding
value for the length $n$ Toda bracket \w{\lra{f\sb{1},\dotsc,f\sb{n}}} is
\w[.]{\varphi'\sb{n-1}\circ\alpha'\sb{n-1}:E(1,0,2,\dotsc,2)\simeq\Sigma\sp{n-2}A\sb{1}\to
  A\sb{n+1}}
If this does not vanish, we must try other choices of the data. If this value is zero,
we choose an extension \w{\psi'\sb{n-1}:E(1,2 ,2,\dotsc, 2)\to A\sb{n+1}}
for \w{\varphi'\sb{n-1}}  in the (homotopy) cofibration sequence
\w[.]{E(1,0,2,\dotsc,2)~\xra{\alpha'\sb{n-1}}~E(1,1,2,\dotsc,2)~\epic~E(1,2,2,\dotsc,2)}
In fact, this map \w{\psi'\sb{n-1}} is precisely \w[,]{\varphi\sb{n}}  the second piece of input
needed to define our value for \w[.]{\lra{f\sb{0},\dotsc,f\sb{n}}}
\end{remark}

\begin{mysubsection}{The case $n=2$}\label{scntw}
  In the first step of our inductive process, for \w[.]{n=2} there is no obstruction
  to strictification, and our cofibrant enhanced strictification
  \w{\wD:I\sp{2}\sb{2}\to\C} is obtained in three steps:

\begin{enumerate}
\renewcommand{\labelenumi}{(\alph{enumi})~}
\item We choose a cofibrant replacement for \w[,]{A\sb{0}} and represent \w{[f\sb{0}]} by
a cofibration \w[,]{f\sb{0}:A\sb{0}\to A\sb{1}} and \w{[f\sb{1}]} by
any map \w[.]{f\sb{1}:A\sb{1}\to A\sb{2}}
\item  Since \w[,]{[f\sb{1}]\circ[f\sb{0}]=0} in \w[,]{\ho\C} we can choose a nullhomotopy
  \w{F:CA\sb{0}\to A\sb{2}} for \w[,]{f\sb{1}]\circ[f\sb{0}} making the outer square in the following
  diagram commute, where the inner square is the (homotopy) pushout (which is the homotopy cofiber
  of   \w[):]{f\sb{0}}
\mydiagram[\label{eqtwosquare}]{
  A\sb{0} \ar@{^{(}->}[d]_{f\sb{0}}  \ar@{^{(}->}[rr]^{i} &&
  CA\sb{0} \ar@{^{(}->}[d]^{j} \ar@/^2em/[ddr]^{F} \\
  A\sb{1}  \ar@{^{(}->}[rr]_{k} \ar@/_2em/[drrr]_{f\sb{1}} && \Cof{f\sb{0}} \ar@{-->}[rd]^{\xi} \\
  &&& A\sb{2}
}
\item Changing $\xi$ into a cofibration \w{\xi':\Cof{f\sb{0}}\hra \wA\sb{2}} and precomposing with
  the given structure maps $j$ and $k$ yields the required cofibrant  \w[.]{\wD:I\sp{2}\sb{2}\to\C}
\end{enumerate}

This description makes it clear that even though the homotopy
classes of nullhomotopies $F$ for \w{f\sb{1}\circ f\sb{0}} are in
bijection with \w{[\Sigma A\sb{0},\,A\sb{2}]} (see \cite[\S 1]{SpanS}),
the homotopy type of possible  (enhanced)
strictifications \w{\wD} (given \w{[f\sb{0}]} and \w[)]{[f\sb{1}]}
is completely determined by \w[.]{[\xi]\in[\Cof{f\sb{0}},\,A\sb{2}]}
Thus, perhaps surprisingly, the homotopy classification of strictifications has
less information than the choices of nullhomotopies.

The extended diagram \w{E:I\sp{3}\sb{3}\to\C} is then the solid \w{3\times 3} square in
\mydiagram[\label{eqthreesquare}]{
  \stackrel{\mbox{$A\sb{0}$}}{\mbox{\scriptsize{(00)}}} \ar@<2ex>@{^{(}->}[rr]^{i}
  \ar@{^{(}->}[d]_{f\sb{0}}    &&
    \stackrel{\mbox{$CA\sb{0}$}}{\mbox{\scriptsize{(01)}}} \ar@<2ex>@{->>}[rr]
    \ar@<1ex>@{^{(}->}[d]^{F} &&
    \stackrel{\mbox{$\tS A\sb{0}$}}{\mbox{\scriptsize{(02)}}}
    \ar@<1ex>@{^{(}->}[d]^{\alpha\sb{2}} && \\
    \stackrel{\mbox{$A\sb{1}$}}{\mbox{\scriptsize{(10)}}}
    \ar@<2ex>@{^{(}->}[rr]^{f\sb{1}} \ar@{->>}[d] &&
  \stackrel{\mbox{$A\sb{2}$}}{\mbox{\scriptsize{(11)}}}   \ar@<2ex>@{->>}[rr]
  \ar@<1ex>@{->>}[d]_(0.40){\simeq}^(0.55){q} \ar@/_2em/@{-->}[ddrrrr]_{f\sb{2}} &&
  \stackrel{\mbox{$C\sb{f\sb{1}}$}}{\mbox{\scriptsize{(12)}}} \ar@{->>}[d]
  \ar@/^2em/@{-->}[ddrr]^{\varphi\sb{2}} && \\
\stackrel{\mbox{$C\sb{f\sb{0}}$}}{\mbox{\scriptsize{(20)}}} \ar@<2ex>@{^{(}->}[rr]^{h\sb{0}}
&& \stackrel{\mbox{$\tA\sb{2}$}}{\mbox{\scriptsize{(21)}}} \ar@<2ex>@{->>}[rr] &&
\stackrel{\mbox{$C\sb{h\sb{0}} =C\sb{\alpha\sb{0}} $}}{\mbox{\scriptsize{(22)}}}
\ar@{.>}[drr]^{\exists?\psi\sb{2}} &&\\
  && && && A\sb{3}
}
\noindent (cf.\ the dual diagram in \cite[(0.3)]{BSenH}).

Here (and later in the paper) \w{\tS A} denotes the version of the suspension of $A$
obtained as the cofiber of \w[,]{A\hra CA} as in the first row of \wref[.]{eqthreesquare}
\end{mysubsection}

\begin{remark}\label{rcofsq}
  We note that more generally if we start with the pushout $P$ of two
  cofibrations \w{f:X\hra Y} and \w{k:X\hra Z} in a pointed model category,
  as in \wref[,]{eqtwosquare} the resulting extended diagram takes the form:
\mydiagram[\label{eqtwosquares}]{
  X \ar@{^{(}->}[d]_{h}  \ar@{^{(}->}[rr]^{f} && Y \ar@{^{(}->}[d]^{i} \ar[rr] && C\sb{f} \ar[d]^{=} \\
  Z \ar[d]  \ar@{^{(}->}[rr]^{j} && P \ar[d] \ar[rr] && C\sb{j} \ar[d] \\
C\sb{h} \ar[rr]^{=} && C\sb{i} \ar[rr] && \ast
}
\noindent If we have a further cofibration \w[,]{\xi:P\hra W} as in \S \ref{scntw}(c), and let
\w{g:=\xi\circ j:Z\to W} and \w[,]{k:=\xi\circ i:Y\to W}  then taking
iterated pushouts in the following solid diagram
\mydiagram[\label{eqitpushouts}]{
  Z \ar[d]  \ar@{^{(}->}[rr]^{j} \ar@/^2em/@{^{(}->}[rrrr]^{g} && 
  P \ar@{.>}[d] \ar@{^{(}->}[rr]^{\xi} && W \ar@{.>}[d] \\
\ast \ar@{^{(}.>}[rr] && C\sb{j} \ar@{^{(}.>}[rr]^{\ell} && Q 
}
\noindent we see that the map $\ell$ is again a cofibration, by cobase change, and that
$Q$ is in fact \w{C\sb{g}} (the cofiber of $g$), since the large rectangle in
\wref{eqitpushouts} is a also pushout. Thus from \wref{eqtwosquares} we deduce that the induced map
\w{\ell:C\sb{f}\to C\sb{g}} is a cofibration (and similarly for \w[).]{m:C\sb{h}\to C\sb{k}} 

In summary, whenever our initial square \w{D:I\sp{n}\sb{2}\to\C} is cofibrant in the
model category structure of \S \ref{retodadiag}, all rows and columns in the
extended diagram \w{E:I\sp{n}\sb{3}\to\C} are strict cofibration sequences.
\end{remark}

\begin{mysubsection}{The ordinary Toda bracket}\label{sotb}
Our formalism indicates that \w{\varphi\sb{2}=f\sb{1}} and
\w[,]{\alpha\sb{2}=f\sb{1}} so the ``Toda bracket
\w{\lra{f\sb{0},f\sb{1}}} of length $2$'' is the composite
\w[,]{f\sb{1}\circ f\sb{0}} which is nullhomotopic by assumption.
A choice of nullhomotopy determines the extension
\w[,]{\psi\sb{2}=h\sb{0}} and in fact setting
\w[,]{\widetilde{f}\sb{1}:=q\circ f\sb{1}:A\sb{1}\to\tA\sb{2}} we
obtain a strictification for the initial length $2$ subdiagram of
\wref[,]{eqtodadiag} as implied by Lemma \ref{lhtb}.

Now we add the map \w{f\sb{2}:A\sb{2}\to A\sb{3}} to the solid
square in \wref[:]{eqthreesquare} since \w{f\sb{2} \circ f\sb{1}\sim 0} and the
middle row is a (homotopy) cofibration sequence, we have an induced map
\w[,]{\varphi:C\sb{f\sb{1}}\to A\sb{3}} and we define the \emph{value}
of the Toda bracket (of length $2$) associated to these choices to be the
homotopy class of \w[.]{\varphi\circ g\sb{0}:\tS A\sb{0}\to A\sb{3}}
Since the right column is again a (homotopy)
cofibration sequence, $\varphi$ extends to $\psi$ as indicated if and only if
\w[.]{\varphi\circ g\sb{0}\sim 0}
In this case we see that \w{\widetilde{f}\sb{2}:=\psi\circ r:\tA\sb{2}\to A\sb{3}}
represents \w[,]{f\sb{2}} and we also have \w{\widetilde{f}\sb{2}\circ\widetilde{f}\sb{1}=0}
\wwh in other words, we can extend
the above to a strictification of the initial length $3$ subdiagram of \wref{eqtodadiag} if and only if
\w{\varphi\circ g\sb{0}\sim 0} for some choice of $\xi$ and $\varphi$\vsm.
\end{mysubsection}

\begin{mysubsection}{The case $n=3$}\label{scnth}
Here the extended diagram is described by the solid \w{3\times 3\times 3} cube:
$$
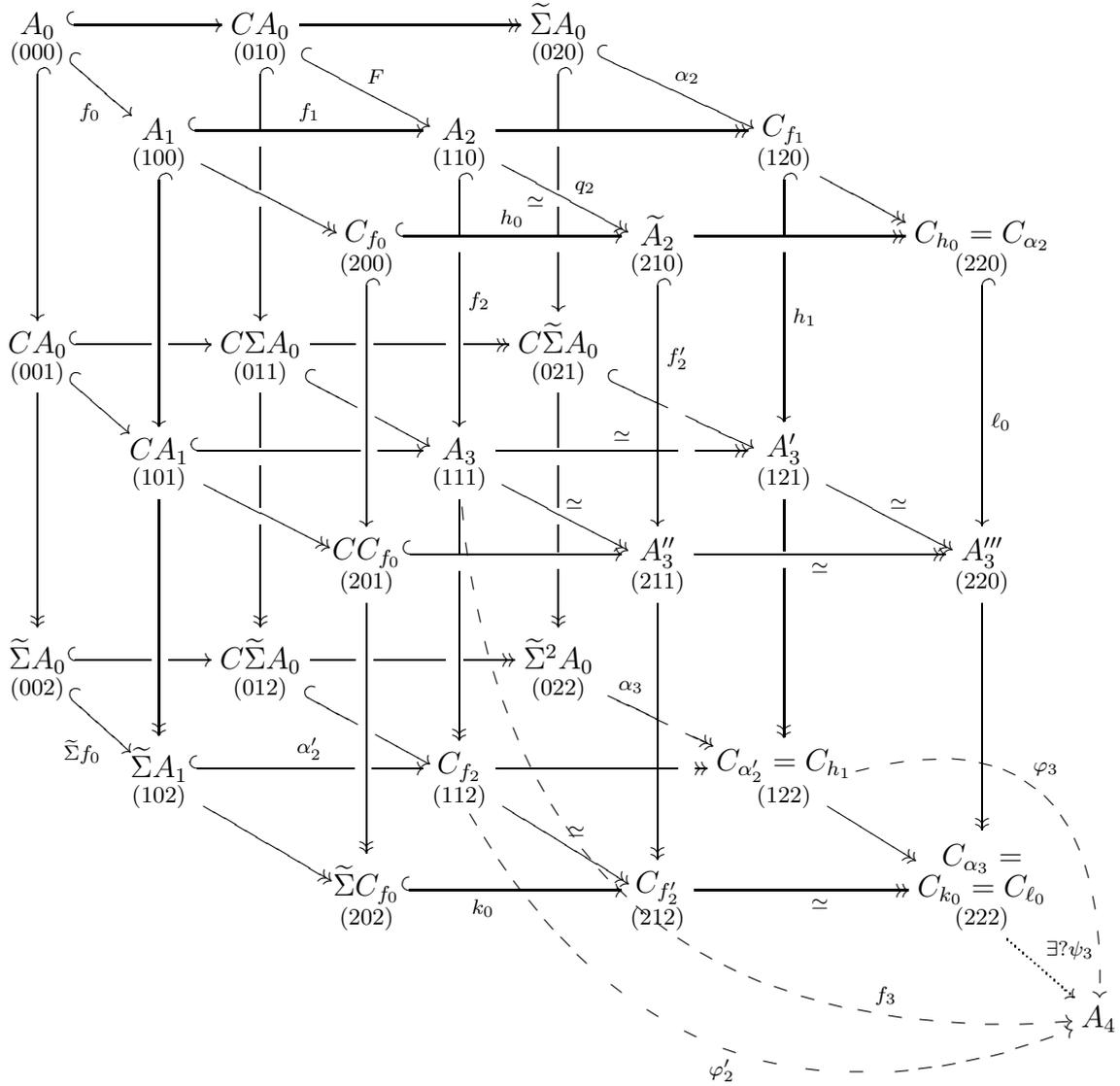
\xymatrix@R=10pt@C=5pt{
  %
  %
  \stackrel{\mbox{$A\sb{0}$}}{\mbox{\scriptsize{(000)}}} \ar@<2ex>@{^{(}->}[rrr]
  \ar@<1ex>@{^{(}->}[rrd]_{f\sb{0}}  \ar@{^{(}->}[ddd] &&&
  \stackrel{\mbox{$CA\sb{0}$}}{\mbox{\scriptsize{(010)}}} \ar@<2ex>@{->>}[rrr]
  \ar@<1ex>@{^{(}->}[rrd]^{F} \ar@{^{(}->}|(.27){\hole}|(.46){\hole}[ddd] &&&
  \stackrel{\mbox{$\tS A\sb{0}$}}{\mbox{\scriptsize{(020)}}}
  \ar@<1ex>@{^{(}->}[rrd]^{\alpha\sb{2}}  \ar@{^{(}->}|(.27){\hole}|(.46){\hole}|(.60){\hole}[ddd] && \\
  &&   \stackrel{\mbox{$A\sb{1}$}}{\mbox{\scriptsize{(100)}}} \ar@<2ex>@{^{(}->}[rrr]^{f\sb{1}}
  \ar@<1ex>@{->>}[rrd] \ar@{^{(}->}[ddd] &&&
  \stackrel{\mbox{$A\sb{2}$}}{\mbox{\scriptsize{(110)}}}   \ar@<2ex>@{->>}[rrr]
  \ar@<1ex>@{->>}[rrd]_(0.40){\simeq}^(0.55){q\sb{2}}
  \ar@{^{(}->}|(.27){\hole}[ddd]^(0.45){f\sb{2}} &&&
  \stackrel{\mbox{$C\sb{f\sb{1}}$}}{\mbox{\scriptsize{(120)}}} \ar@{->>}[rrd]
  \ar@{^{(}->}|(.27){\hole}[ddd]^{h\sb{1}} \\
  &&&&
\stackrel{\mbox{$C\sb{f\sb{0}}$}}{\mbox{\scriptsize{(200)}}} \ar@<2ex>@{^{(}->}[rrr]^{h\sb{0}}
\ar@{^{(}->}[ddd]
&&& \stackrel{\mbox{$\tA\sb{2}$}}{\mbox{\scriptsize{(210)}}} \ar@<2ex>@{->>}[rrr]
\ar@{^{(}->}[ddd]^(0.3){f'\sb{2}} &&&
\stackrel{\mbox{$C\sb{h\sb{0}} =C\sb{\alpha\sb{2}} $}}{\mbox{\scriptsize{(220)}}}
\ar@{^{(}->}[ddd]^{\ell\sb{0}} \\
%
%
\stackrel{\mbox{$CA\sb{0}$}}{\mbox{\scriptsize{(001)}}} \ar@<2ex>@{^{(}->}|(.55){\hole}[rrr]
  \ar@<1ex>@{^{(}->}[rrd]  \ar@{->>}[ddd]  &&&
  \stackrel{\mbox{$C\Sigma A\sb{0}$}}{\mbox{\scriptsize{(011)}}}
  \ar@<2ex>@{->>}|(.35){\hole}|(.65){\hole}[rrr]
  \ar@<1ex>@{^{(}->}|(.50){\hole}[rrd] \ar@{->>}|(.27){\hole}|(.46){\hole}[ddd] &&&
  \stackrel{\mbox{$C\tS A\sb{0}$}}{\mbox{\scriptsize{(021)}}}
  \ar@<1ex>@{^{(}->}|(.50){\hole}[rrd] \ar@{->>}|(.27){\hole}|(.46){\hole}|(.60){\hole}[ddd] && \\
  &&   \stackrel{\mbox{$CA\sb{1}$}}{\mbox{\scriptsize{(101)}}}
  \ar@<2ex>@{^{(}->}|(.70){\hole}[rrr] \ar@<1ex>@{->>}[rrd] \ar@{->>}[ddd] &&&
  \stackrel{\mbox{$A\sb{3}$}}{\mbox{\scriptsize{(111)}}}   \ar@<2ex>@{->>}|(.70){\hole}[rrr]^{\simeq}
  \ar@<1ex>@{->>}[rrd]^{\simeq} \ar@{->>}|(.27){\hole}[ddd]
  \ar@/_7em/@{-->}[dddddrrrrrr]^(0.8){f\sb{3}} &&&
  \stackrel{\mbox{$A'\sb{3}$}}{\mbox{\scriptsize{(121)}}} \ar@<1ex>@{->>}[rrd]^{\simeq}
  \ar@{->>}|(.26){\hole}[ddd]   &&&\\
&&&& \stackrel{\mbox{$CC\sb{f\sb{0}}$}}{\mbox{\scriptsize{(201)}}} \ar@<2ex>@{^{(}->}[rrr]
\ar@{->>}[ddd]
&&& \stackrel{\mbox{$A''\sb{3}$}}{\mbox{\scriptsize{(211)}}}
\ar@<2ex>@{->>}[rrr]_{\simeq} \ar@{->>}[ddd] &&&
\stackrel{\mbox{$A'''\sb{3}$}}{\mbox{\scriptsize{(220)}}} \ar@{->>}[ddd] \\
%
%
\stackrel{\mbox{$\tS A\sb{0}$}}{\mbox{\scriptsize{(002)}}} \ar@<2ex>@{^{(}->}|(.55){\hole}[rrr]
  \ar@<1ex>@{^{(}->}[rrd]_{\tS f\sb{0}}  &&&
  \stackrel{\mbox{$C\tS A\sb{0}$}}{\mbox{\scriptsize{(012)}}}
  \ar@<2ex>@{->>}|(.35){\hole}|(.65){\hole}[rrr]  \ar@<1ex>@{^{(}->}|(.50){\hole}[rrd] &&&
  \stackrel{\mbox{$\tS\sp{2} A\sb{0}$}}{\mbox{\scriptsize{(022)}}}
  \ar@<1ex>@{->>}|(.50){\hole}[rrd]^(0.25){\alpha\sb{3}} && \\
  &&   \stackrel{\mbox{$\tS A\sb{1}$}}{\mbox{\scriptsize{(102)}}}
  \ar@<2ex>@{^{(}->}|(.68){\hole}[rrr]^{\alpha'\sb{2}}   \ar@<1ex>@{->>}[rrd]  &&&
  \stackrel{\mbox{$C\sb{f\sb{2}}$}}{\mbox{\scriptsize{(112)}}}
  \ar@<2ex>@{->>}|(.70){\hole}[rrr]  \ar@<1ex>@{->>}[rrd]^{\simeq}
  \ar@/_5em/@{-->}[ddrrrrrr]_(0.5){\varphi'\sb{2}}  &&&
  \stackrel{\mbox{$C\sb{\alpha'\sb{2}}=C\sb{h\sb{1}}$}}{\mbox{\scriptsize{(122)}}}
  \ar@<1ex>@{->>}[rrd]  \ar@/^5em/@{-->}[ddrrr]^(0.5){\varphi\sb{3}}  \\
&&&&
  \stackrel{\mbox{$\tS C\sb{f\sb{0}}$}}{\mbox{\scriptsize{(202)}}}
  \ar@<2ex>@{^{(}->}[rrr]_(0.4){k\sb{0}}
&&& \stackrel{\mbox{$C\sb{f'\sb{2}}$}}{\mbox{\scriptsize{(212)}}}
\ar@<2ex>@{->>}[rrr]_{\simeq} &&&
\stackrel{\stackrel{\mbox{$C\sb{\alpha\sb{3}}=$}}
  {\mbox{$C\sb{k\sb{0}}=C\sb{\ell\sb{0}}$}}}{\mbox{\scriptsize{(222)}}}
\ar@{.>}[dr]^{\exists?\psi\sb{3}} &\\
%
%
&&&& &&& &&&  & A\sb{4}
}
$$
\end{mysubsection}

We may summarize the results of this section in the following

\begin{thm}\label{trect}
Assume given a Toda diagram \wref{eqtodadiag} in a model category $\C$ as in \S
\ref{dtodadiag}.
\begin{enumerate}
\renewcommand{\labelenumi}{(\alph{enumi})~}
\item For any representatives \w{f\sb{0}} and \w[,]{f\sb{1}} a choice of
  nullhomotopy \w{F:f\sb{1}\circ f\sb{0}\sim\ast} yields an extended diagram
\w{E:I\sp{3}\sb{3}\to\C} as in \S \ref{scntw}. One can strictify the initial
segment of \wref{eqtodadiag} of length $3$ if and only there is a value of the
Toda bracket which vanishes (for some choice of $F$ and \w{\varphi\sb{2}}
in in \w[).]{eqthreesquare} 
\item Any strictification of the initial segment of \wref{eqtodadiag} of length
  $k$ is equivalent (up to weak equivalence of diagrams) to a
  enhanced strictification \w{\wD:I\sp{k+1}\sb{2}\to\C} as in \S \ref{detodadiag}
\item Given such a (cofibrant) enhanced strictification, with the corresponding
  extended diagram \w{E:I\sp{k+1}\sb{3}\to\C} as in \S \ref{dextodadiag},
  the given strictification extends to one of initial segment of \wref{eqtodadiag}
  of length \w{k+1} if and only if there is a choice of data for
  \w{\lra{f\sb{0},\dotsc,f\sb{k+1}}} for which the Toda bracket vanishes (i.e.,
  has value $0$). 
\end{enumerate}
\end{thm}

%
%
\sect{Toda brackets and filtered objects}
\label{cgtb}

In stable model categories $\C$ (cf.\ \cite{HPStrA}), an
alternative description of Toda brackets, in terms of filtered
objects in $\C$, is available. This involves continuing the
extended diagram one step further in each direction.

We should remark that there is no ``standard'' definition of (long) Toda brackets,
even stably (see \cite{CFranH,GWalkL} and the discussion in \cite{OOshU}).

\begin{mysubsection}{Forward cubes}\label{sfc}
More formally, we assume given a Toda diagram
\begin{myeq}\label{eqlongertd}
A\sb{-1}~\xrightarrow{[f\sb{-1}]}~A\sb{0}~\xrightarrow{[f\sb{0}]}~\dotsc~\to~
A\sb{n}~\xrightarrow{[f\sb{n}]}~A\sb{n+1}
\end{myeq}
\noindent of length \w{n+2} as in \wref[,]{eqtodadiag} together with a rectification
\w{D:I\sp{n}\sb{\md}=I\sp{n}\sb{2}\to\C} of the central segment
\w{A\sb{0}\to\dotsc\to A\sb{n}} of length $n$, with cofibrant replacement
\w{\wD:I\sp{n}\sb{2}\to\C} and corresponding extended diagram \w{E:I\sp{n}\sb{3}\to\C}
as in \S \ref{dextodadiag} (see diagram \ref{scnth} above).

 Now consider the ``forward'' $n$-dimensional  \wwb{2\times\dotsc\times2}subcube
 \w{I\sp{n}\sb{\fwd}\cong I\sp{n}\sb{2}} of \w[,]{I\sp{n}\sb{3}} consisting of those vertices
 labelled by \w{J=(\vare\sb{1},\dotsc,\vare\sb{n})} with each \w[.]{\vare\sb{i}\in\{1,2\}} Let
  \w{\wE:I\sp{n}\sb{\fwd}\to\C} be a cofibrant replacement for \w[.]{E\rest{I\sp{n}\sb{\fwd}}}
We have a new forward \wwb{3\times\dotsc\times 3}cube
\w[,]{\wI\sp{n}\sb{3}} with vertices now labelled by
\w{J=(\vare\sb{1},\dotsc,\vare\sb{n})} with
\w[,]{\vare\sb{i}\in\{1,2,3\}} and  we see that $\wE$ be extended
to the \emph{extended forward diagram} \w{F:\wI\sp{n}\sb{3}\to\C}
by taking strict cofibration sequences in each direction, as in \S
\ref{dextodadiag}.

Moreover, Lemma \ref{lextodadiag} extends in the obvious way to $F$, with
$$
F(\vare\sb{1}\dotsc\vare\sb{i-1},3,\vare\sb{i+1}\dotsc\vare\sb{n})~\simeq~
\Sigma E(\vare\sb{1}\dotsc\vare\sb{i-1},0,\vare\sb{i+1}\dotsc\vare\sb{n})
$$
\noindent for any \w{\vare\sb{j}\in\{1,2,3\}} \wb[.]{j= 1, \dotsc,i-1,i+1,\dotsc,n}

Furthermore, if $\C$ is a stable model category (see \cite[Ch.\ 7]{HovM}), $E$ can actually be
recovered from $F$, up to weak equivalence,  by taking homotopy fibers of the forward
generating cube \w{F\rest{I\sp{n}\sb{\fwd}}=\wE\rest{I\sp{n}\sb{\fwd}}} along each edge.

Now let \w[,]{J\sp{k}:=(1\dotsc 12\dotsc 2)} with \w{n-k+1} digits 1 followed by
\w{k-1} digits 2 \wb[,]{1\leq k\leq n} and write \w[.]{X\sp{k}:=\wE(J\sp{k})=F(J\sp{k})}
Note that by Lemma \ref{lextodadiag} we have homotopy
cofibration sequences of the form
\begin{myeq}\label{eqpxcofseq}
  \Sigma\sp{k-1}A\sb{n-k}\simeq E(\mbox{$\displaystyle{\underbrace{1\dotsc 1}_{n-k}}$}\,0\,
    \mbox{$\displaystyle{\underbrace{2\dotsc 2}_{k-1}}$}) ~
    \xra{g\sb{k}}~E(J\sp{k})~\xra{i\sb{k}}~E(J\sp{k+1})~,
\end{myeq}
\noindent so for $F$ we have homotopy cofibration sequences
\begin{myeq}\label{eqxcofseq}
  F(J\sp{k})~\xra{i\sb{k}}~F(J\sp{k+1})~\xra{r\sb{k}}~
  F(\mbox{$\displaystyle{\underbrace{1\dotsc 1}_{n-k}}$}\,3\,
  \mbox{$\displaystyle{\underbrace{2\dotsc 2}_{k-1}}$}) \simeq\Sigma\sp{k}A\sb{n-k+1}
\end{myeq}
\noindent for each \w[.]{0\leq k<n} By convention we set \w[;]{F(J\sp{0}):=\ast} then
\w[,]{F(J\sp{1})=F(1\dotsc 1)\simeq A\sb{n}} so \wref{eqxcofseq} is still a homotopy
cofibration sequence for \w[.]{k=0} We write $X$ for \w[,]{F(J\sp{n})} and denote the composite
cofibration \w{i\sb{n-1}\circ\dotsc\circ i\sb{1}:F(J\sp{1})\hra F(J\sp{n})} by \w[.]{j\sb{X}:A\sb{n}\to X}
\end{mysubsection}

\begin{remark}\label{rfwdiag}
Note that the composite
\w{r\sb{k}\circ g\sb{k+1}:\tS\sp{k}A\sb{n-k}\to\tS\sp{k}A\sb{n-k+1}} represents
\w{\tS\sp{k}f\sb{n-k}} (because of the functoriality of the extension of $\wD$ to $E$ and
then to $F$).
\end{remark}

\begin{lemma}\label{lfwdiag}
  Given a Toda diagram \wref{eqlongertd} of length \w[,]{n+2} the data actually
  needed to specify a value of  the Toda bracket \w{\lra{f\sb{-1},\dotsc,f\sb{n}}} of
  length \w{n+1} consists of:
\begin{enumerate}
\renewcommand{\labelenumi}{(\alph{enumi})~}
\item The extended forward diagram \w{F:\wI\sp{n}\sb{3}\to\C} associated to the
  strictification of \w[;]{A\sb{0}\to\dotsc\to A\sb{n}}
\item A map \w{\alpha\sb{n+1}:\Sigma\sp{n}A\sb{-1}\to X:=F(J\sp{n})} such that
  \w[.]{r\sb{n}\circ\alpha\sb{n+1}\sim \Sigma\sp{n}f\sb{-1}}
\item A map \w{\varphi\sb{n+1}:X\to A\sb{n+1}} such that
  \w[.]{\varphi\sb{n+1}\circ j\sb{X}\sim f\sb{n}}
\end{enumerate}
\noindent The associated value is the class of
\w[.]{\varphi\sb{n+1}\circ\alpha\sb{n+1}:\Sigma\sp{n}A\sb{-1}\to A\sb{n+1}}
\end{lemma}

\begin{proof}
  This follows from Definition \ref{dhtb}, once we notice that the $n$-dimensional
  \wwb{3\times\dotsc\times 3}cube \w{I\sp{n}\sb{3}} which indexes diagram $\wD$ of \S \ref{sfc}
  is actually the middle cube \w{I\sp{n}\sb{\md}} of Remark \ref{rhtb} with respect to
  the Toda diagram \wref[,]{eqlongertd} and thus \w{X=F(2\dotsc 2)} is in fact \w{E(12\dotsc 2)}
  (with $n$ digits 2) for the full \wwb{3\times\dotsc\times 3}cube diagram
  \w{E:I\sp{n+1}\sb{3}\to C} of \wref[.]{eqlongertd}
  However, from the description in \S \ref{dhtb} we see that the only
  part of the enhanced strictification $E$ needed to compute the Toda bracket and not determined by
  $F$ is the map \w[.]{\alpha\sb{n+1}:\tS\sp{n}A\sb{-1}\to E(12\dotsc 2)} The fact that
  \w{r\sb{n}\circ\alpha\sb{n+1}\sim \Sigma\sp{n}f\sb{-1}} follows by continuing the cofibration
  sequence \wref{eqpxcofseq} for \w{k=n+1} one step to the right, as in \wref[.]{eqxcofseq}

  The other ingredient needed is the precisely the map \w{\varphi\sb{n+1}:E(12\dotsc 2)\to A\sb{n+1}}
  associated to the vanishing of the right Toda bracket \w[.]{\lra{f\sb{0},\dotsc,f\sb{n}}}
  The fact that \w{\varphi\sb{n+1}\circ j\sb{X}\sim f\sb{n}} can be read off from Lemma \ref{lextodadiag}
  and the description of \w{\varphi\sb{n}} in \S \ref{dhtb}.
\end{proof}

\begin{example}\label{egtoda}
As Toda showed in \cite[Proposition 5.6]{TodC}, the generator \w{\nu'\in\pi\sb{6}\bS{3}}
(of order $4$) is one value of the Toda bracket for the length $3$ \wb{n=1} diagram:
\begin{myeq}\label{eqnup}
  \bS{5}~\xra{\ett{4}}~\bS{4}~\xra{2\iot{4}}~\bS{4}~\xra{\ett{4}}~\bS{3}~,
\end{myeq}
\noindent where the maps \w{\ett{k}} are suspended Hopf maps (of order $2$). There is
indeterminacy of order $2$, and the other value of the Toda bracket is another generator. 

Thus \w[,]{X\sb{0}=\ast} \w[,]{X\sb{1}=A\sb{1}=\bS{4}} with
\w{X\sb{2}} the cofiber of \w[,]{f\sb{0}=2\iot{4}} the
$5$-dimensional mod $2$ Moore space. The fact that \w{\ett{4}} has
order $2$ implies that it extends to a map
\w[,]{\varphi\sb{2}:X\sb{2}\to\bS{3}} while the fact that
\w{\ett{6}} also has order $2$ implies  that it factors through
\w{X\sb{2}} in the continued cofibration sequence
\w{X\sb{2}\to\bS{5}\xra{2}\bS{5}} (since it is in the
stable range), yielding \w[.]{\alpha\sb{2}:\bS{6}\to X\sb{2}}

The composite \w{\bS{6}=\Sigma A\sb{0}~\xra{\varphi\sb{2}\circ\alpha\sb{2}}~A\sb{2}=\bS{3}}
is the required value \w{\nu'} of \w[.]{\lra{\ett{3},2\iot{4},\ett{4}}}
\end{example}

\begin{mysubsect}{Generalized Toda brackets associated to a filtered object}
\label{sgtbfo}

We do not in fact need to be given a Toda diagram \wref{eqtodadiag} in order to use this
approach to defining Toda brackets \wh all we need are cofibration sequences
as in \wref[.]{eqxcofseq}

Thus assume that for each \w{-1\leq k<\ell}  we have a homotopy cofibration sequence
\begin{myeq}\label{eqexcofseq}
  X\sb{k}~\xra{j\sb{k}}~X\sb{k+1}~\xra{r\sb{k}}~C\sb{k+1}~\xra{\delta\sb{k}}~
  \Sigma X\sb{k}~\dotsc
\end{myeq}
\noindent in some pointed model category $\C$, starting with \w[,]{X\sb{-1}=\ast} so
\w[.]{X\sb{0}=C\sb{0}} We think of \w{X:=X\sb{\ell}} as an object in $\C$ \emph{filtered by}
\w[,]{X\sb{0}\xra{j\sb{0}} X\sb{1}\xra{j\sb{1}} \dotsc \xra{j\sb{\ell-1}}X\sb{\ell}} with
\emph{filtration quotients} \w{C\sb{k}} \wb[.]{k=0,\dotsc,\ell} Again we denote the composite cofibration
\w{j\sb{n-1}\circ\dotsc\circ j\sb{0}:X\sb{0}\hra X\sb{\ell}} by \w[.]{j\sb{X}:C\sb{0}\to X}

The cofibration sequence \wref{eqexcofseq} might extend to the left, presenting
\w{X\sb{k+1}} as the homotopy cofiber of a map \w{g\sb{k}:B\sb{k}\to X\sb{k}}
  with \w{C\sb{k+1}\simeq\Sigma B\sb{k}} and \w[,]{\delta\sb{k}\sim\Sigma g\sb{k}}
  as in \wref[.]{eqxcofseq} This need not always exist, but we shall use this
  notation when we have such a map \w[.]{g\sb{k}}

We then think of the filtered space $X$ as a template for a \emph{generalized Toda
  bracket of length} \w[.]{\ell+1} The \emph{data} for the bracket consist of any
two homotopy classes of maps \w{\alpha\sb{\ell+1}:W\to X} and
\w{\varphi\sb{\ell+1}:X\to Z} (where $W$ and $Z$ are arbitrary, though often $W$
is a sphere). The \emph{value} associated to this data is the homotopy class
of the composite \w{\varphi\sb{\ell+1}\circ\alpha\sb{\ell+1}} in \w[.]{[W,Z]}
\end{mysubsect}

\begin{defn}\label{dfiltoda}
Given homotopy cofibration sequences as in \wref[,]{eqexcofseq} we denote
\w{\Sigma r\sb{k-1}\circ\delta\sb{k}:C\sb{k+1}\to \Sigma C\sb{k}} by
\w[.]{\gam{k+1}} We also write \w{\Sigma\sp{-1}C\sb{\ell+1}} for
$W$ and \w{\Sigma\sp{\ell+1}C\sb{-1}} for $Z$ as above and define
\w{\Sigma\sp{-1}\gam{\ell+1}:=r\sb{\ell}\circ\alpha\sb{\ell+1}}
and \w[.]{\Sigma\sp{\ell}\gam{0}:=\varphi\sb{\ell+1}\circ j\sb{X}}
\end{defn}

\begin{prop}\label{psttoda}
Given a filtered object $X$ as in \S \ref{sgtbfo}, the sequence
\begin{myeq}\label{eqdesustoda}
\Sigma\sp{-1}C\sb{\ell+1}\xra{\Sigma\sp{-1}\gam{\ell+1}}C\sb{\ell}\xra{\gam{\ell}}
\Sigma C\sb{\ell-1}\xra{\Sigma\gam{\ell-1}}
\Sigma\sp{2}C\sb{\ell-2}\to\dotsc
\Sigma\sp{\ell}C\sb{0}\xra{\Sigma\sp{\ell}\gam{0}}\Sigma\sp{\ell+1}C\sb{-1}
\end{myeq}
\noindent is a Toda diagram of length \w[.]{\ell+2}
Furthermore, when \wref{eqexcofseq} is obtained as in \S \ref{sfc},
\wref{eqdesustoda} is the $\ell$-fold suspension \wref{eqtodadiag} (extended to
as above to \w[),]{A\sb{\ell+1}} with \w{C\sb{k}\simeq\Sigma\sp{k}A\sb{\ell+1-k}} and
\w[.]{\gam{k}\sim\Sigma\sp{k}f\sb{\ell+1-k}}
\end{prop}

\begin{proof}
We have
\begin{myeq}\label{eqadjcompz}
\Sigma\gam{k}\circ\gam{k+1}~\sim~0\hsp\text{for each}\hsm 0<
k<\ell
\end{myeq}
\noindent since the middle factors of the composite are two successive maps
from \wref[.]{eqexcofseq} Moreover, \wref{eqadjcompz} also holds for \w{k=0} and
\w{k=\ell} (for the same reason). When \wref{eqexcofseq} is as in \S \ref{sfc}, the
claim follows from Lemma \ref{lfwdiag}.
\end{proof}

\begin{mysubsection}{CW filtrations}
\label{scwfil}
If \w{\bX\in\Topa} is an $\ell$-dimensional connected CW complex and \wref{eqexcofseq}
is its CW filtration, let \w{g\sb{j}: B\sb{j}\to X\sb{j}} be the attaching map for the
$j$-skeleton \w{X\sb{j}:=\bX\up{j}} of $\bX$, so the cofibration sequence
\w{B\sb{j}\xra{g\sb{j}} X\sb{j} \xra{i\sb{j}} X\sb{j+1}}  defines \w[.]{X\sb{j+1}}
 Thus for \w[,]{j=0,\dotsc\ell-1} \w{B\sb{j}} will be a wedge of $j$-spheres,
 and in the (suspended) Toda diagram \wref{eqdesustoda} associated to the CW
 filtration for $X$, each \w{\Sigma\sp{j}C\sb{\ell-j}} is a wedge of $\ell$-spheres.
Only \w{W=\Sigma\sp{-1}C\sb{\ell+1}} and \w{Z=\Sigma\sp{\ell+1}C\sb{-1}} can be arbitrary.

However, if \w{V\sb{1}\xra{f}V\sb{2}\xra{g}V\sb{3}} are maps between wedges of $\ell$-spheres
and \w[,]{g\circ f\sim 0} this must also hold when we restrict to each wedge summand
of \w{V\sb{1}} and project onto each summand of \w[.]{V\sb{3}}  Thus are reduced to the case
where $g$, as a map from a wedge of $m$ copies of \w{S\sp{\ell}} to \w[,]{S\sp{\ell}}
is given by a vector of integers \w[.]{(a\sb{1},\dotsc, a\sb{m})} Dually $f$ is also given 
by \w[,]{(b\sb{1},\dotsc, b\sb{m})} with \w[.]{\sum\sb{i}\,a\sb{i}b\sb{i}=0}
In this case the associated Toda bracket is trivial.
\end{mysubsection}

This suggests the following

\begin{defn}\label{dspherical}
   A \emph{spherical filtration} \w{X\sb{0}\to X\sb{1}\to\dotsc X\sb{\ell}} of length $\ell$
   on a CW complex \w{X=X\sb{\ell}} is a filtration as above such that
   for each \w[:]{0\leq k\leq \ell}
\begin{enumerate}
\renewcommand{\labelenumi}{(\alph{enumi})~}
\item The filtration quotient \w{C\sb{k}}  is homotopy equivalent to a
  \wwb{c\sb{k}-1}connected  wedge of spheres, (that is, \w{c\sb{k}} is the lowest dimension
  of a sphere in the wedge).
\item We never have two successive maps in \wref{eqdesustoda}
(and thus in \wref[)]{eqtodadiag} between wedges of spheres where the lowest spheres
have the same dimension. That is:
\begin{myeq}\label{eqspherical}
\text{if}\hsm c\sb{k}+1=c\sb{k+1}\hsm \text{then}\hsm c\sb{k-1}+1<c\sb{k}\hsm \text{and}\hsm
c\sb{k+1}+1<c\sb{k+2}
\end{myeq}
\noindent when defined.
\end{enumerate}
Let \w{c\sb{-1}-1} denote the connectivity of $Z$ and \w{c\sb{\ell+1}-1}
the connectivity of $W$, and assume that \wref{eqspherical} holds also for
\w{k=0} and \w[.]{k=\ell} 
Given homotopy classes \w{\alpha\sb{n+1}:W\to X} and
\w[,]{\varphi\sb{\ell+1}:X\to Z} we then say that
\w{\varphi\sb{\ell+1}\circ\alpha\sb{n+1}\in[W,Z]} is a \emph{value of
the spherical Toda bracket} associated to the filtration on $X$.
\end{defn}

\begin{remark}\label{rspherical}
  Using simplicial \Pa resolutions (see \cite[\S 4.2]{StoV}) and the connectivity results
  of \cite[Proposition 4.2.2]{BlaH}, one can show that for each \w[,]{n>r\geq 2}
  any \wwb{r-1}connected space $\bX$ has an $n$-skeleton with a spherical filtration
  of length \w[.]{\ell\leq\lceil\frac{2(n-r+1)}{3}\rceil}
We omit the proof, since this result is not needed for our purposes.
\end{remark}

\begin{example}\label{egfiltration}
  In the stable range the $2$-local spherically filtered \wwb{k+i}skeleta
  of \w{\bX=\bS{k}\lra{k}} \wb{i=1,\dotsc,4} are constructed as follows
  (using Toda's calculations in \cite[Ch. XIV]{TodC}):
\begin{enumerate}
\renewcommand{\labelenumi}{(\alph{enumi})~}
\item We start with \w[,]{C\sb{1}=Z\sb{1}:=\bS{k+1}\vee\bS{k+3}} with
  the covering map \w{p\up{1}:Z\sb{1}\to\bS{k}} given by the Hopf maps
  \w{\eta\sb{k+1}\bot\nu\sb{k+1}} (inducing a surjection in \w{\pi\sb{j}}
  for \w[).]{1\leq j\leq 4} Here \w{\alpha\bot\beta} indicates the map induced
  by the coproduct structure of the wedge. We may omit \w{\bS{k+3}} for
  \w[,]{i\leq 2} since \w{\bS{k+1}} itself is a \wwb{k+1}skeleton for $\bX$.
\item Next, we have the cofibration sequence
$$
  B\sb{1}:=\bS{k+1}\vee\bS{k+3}~
  \xra{2\iot{k+1}\bot(4\iot{k+3}-\iot{k+1}\eta\sb{k+1}\eta\sb{k+2})}~
Z\sb{1}~\xra{j\sb{1}}~Z\sb{2}~\xra{r\sb{1}}~ C\sb{2}=\bS{k+2}\vee\bS{k+4}~.
$$
\noindent Here \w{\gamma\sb{2}:C\sb{2}\to\Sigma C\sb{1}} on the lowest wedge summand
\w{\bS{k+2}} of \w{C\sb{2}} is the suspension of  \w{2\iot{k+1}} on the
\ww{\bS{k+1}}-summand on the right, composed with \w[.]{\Id:\Sigma Z\sb{1}\to\Sigma C\sb{1}}
This is just the order $2$ map.

The covering map \w{p\up{2}:Z\sb{2}\to\bS{k}} is induced from \w{p\up{1}}
by the fact that \w{\eta\sb{k+1}} has order $2$ and
\w[.]{4\nu\sb{k}=\eta\sb{k}\eta\sb{k+1}\eta\sb{k+2}}

Thus \w{Z\sb{2}\up{k+i}}  is already a \wwb{k+i}skeleton of $\bX$ for \w[,]{i\leq 2}
which may be identified with the mod $2$ Moore space \w[.]{\bM{k+1}}  Moreover, since
\begin{myeq}\label{eqpimoore}
   \pi\sb{i}\bM{k+1} ~\cong~\begin{cases}
     \ZZ/2\lra{\xi} & \text{for}\ i=k+1\\
     \ZZ/2\lra{\xi\circ\eta\sb{k+1}} & i=k+2\\
     \ZZ/4\lra{\beta} & i=k+3\\
     \ZZ/2\lra{\beta\circ\eta\sb{k+3}}\oplus\ZZ/2\lra{\xi\circ\nu\sb{k+1}} & i=k+4~.
     \end{cases}
\end{myeq}
\noindent we see that for \w[,]{3\leq i\leq 4} \w{Z\sb{2}} is constructed by wedging
\w{\bM{k+1}} with \w{\bS{k+3}} (yielding a \wwb{k+3}skeleton of $\bX$) and then attaching a
\wwb{k+4}cell along \w[.]{4\iot{k+3}-\xi\circ\eta\sb{k+1}\circ\eta\sb{k+2}}
\item Noting that from \wref{eqpimoore} and the choice of attaching map for \w[,]{\bS{k+3}}
  we find that \w{\pi\sb{k+3}Z\sb{2}\cong\ZZ/16\lra{\delta}} and
  \w[.]{\pi\sb{k+4}Z\sb{2}\cong\ZZ/2\lra{\delta\circ\eta\sb{k+3}}\oplus\ZZ/2
    \lra{\xi\circ\nu\sb{k+1}}}

We therefore have a cofibration sequence
$$
\bS{k+3}\vee\bS{k+4}~\xra{8\delta\bot\xi\circ\nu\sb{k+1}}~
Z\sb{2}~\xra{j\sb{2}}~Z\sb{3}~\xra{r\sb{2}}~ C\sb{3}=\bS{k+4}\vee\bS{k+5}~,
$$
\noindent with the covering map \w{p\up{3}:Z\sb{3}\to\bS{k}} induced from \w[,]{p\up{2}}
yielding a \wwb{k+4}skeleton for $\bX$.

In this case, \w{\gamma\sb{3}} from the lowest wedge summand \w{\bS{k+4}} of \w{C\sb{3}}
is the suspension of  \w[,]{8\delta}  composed with the pinch map
\w[,]{\Sigma\bM{k+1}\to\bS{k+3}} which is \w[,]{\eta\sb{k+3}} by \wref[.]{eqpimoore}
\item For the \wwb{k+5} and \wwb{k+6}skeleta we must kill \w{\pi\sb{k+4} Z\sb{3}} and then
  \w[.]{\pi\sb{k+5} Z\sb{4}}
\end{enumerate}
\end{example}

\end{document}